\documentclass[12pt,a4paper]{amsart}
\usepackage[a4paper,centering]{geometry}
\usepackage{mathptmx}
\DeclareMathAlphabet{\mathcal}{OMS}{cmsy}{m}{n} 

\usepackage{authoraftertitle,enumerate,fancyhdr,perpage}

\theoremstyle{plain} 
\newtheorem*{theorem}{Theorem}

\MakePerPage[2]{footnote}

\DeclareMathOperator{\End}{End}
\DeclareMathOperator{\FEnd}{FEnd}
\DeclareMathOperator{\Hom}{Hom}
\DeclareMathOperator{\Tr}{Tr}

\begin{document}

\title{A variant of Baer's theorem}
\author{Pasha Zusmanovich}
\address{University of Ostrava, Czech Republic}
\email{pasha.zusmanovich@osu.cz}
\date{First written September 15, 2022; last minor revision March 26, 2023}
\thanks{Rocky Mountain J. Math., to appear} 


\begin{abstract}
We provide a variant of Baer's theorem about isomorphism of endomorphism rings 
of vector spaces over division rings, where the full endomorphism rings are 
replaced by some subrings of finitary maps.
\end{abstract}

\vspace*{-0.8cm}
\maketitle

\pagestyle{fancy}
\setlength{\headheight}{0.5cm}

\fancyhead[L]{\footnotesize P. Zusmanovich}
\fancyhead[C]{\footnotesize \MyTitle}

\fancyhead[R]{\footnotesize \thepage{}}
\fancyfoot[L,C,R]{}
\renewcommand{\headrulewidth}{0pt}

Under Baer's theorem we mean the following result, formulated for the first 
time as Theorem 1 in \cite[Chapter V, \S 4]{baer}: if $V$ and $W$ are two 
(say, right) vector spaces over division rings $D$ and $F$ respectively, then 
the isomorphism of their endomorphism rings $\End_D(V)$ and $\End_F(W)$ implies
a semilinear isomorphism of $V$ and $W$. Baer's original proof uses properties 
of idempotents in the endomorphism ring; for a streamlined and modern exposition
see \cite[Theorem 4.1]{krylov}. Another proof offered in 
\cite[Theorem 8.1]{wolfson} (Ph.D. thesis under Baer) uses Jacobson's density 
theorem. In \cite{racine} and \cite{balaba} Baer's theorem is extended to the 
cases of super vector spaces and graded vector spaces over super or graded 
division rings, respectively. Later on Baer's theorem was also extended to 
modules over large classes of abelian groups (the so-called Baer--Kaplansky 
theorem), see \cite{krylov} and references therein for a survey.

If $V$ is infinite-dimensional over $D$, the ring $\End_D(V)$ is huge, and one 
may wonder whether in the formulation of Baer's theorem it can be replaced by
something smaller. Somewhat anachronistically, this was done, with small 
variations, already before Baer, in 1940s, in the works of Dieudonn\'e, 
Jacobson, and others, see \cite[Chapter IV, \S 11]{jacobson-rings} or 
\cite[Chapter IX, §11, Theorem 7]{jacobson}: the condition involves isomorphism
of rings of \emph{finitary} linear maps, and the proof uses again Jacobson's 
density theorem. The origin of these works seems to be in similar results 
established earlier in the analytic setting, for rings of bounded operators on 
Banach or Hilbert spaces, or rings of continuous operators on normed spaces, by 
Eidelheit, Mackey, and others; see historical references at 
\cite[p.~94]{jacobson-rings}.

Here we offer another variation on this topic. We retain a relatively narrow
context of vector spaces over division rings, and -- similarly to Jacobson and 
others -- instead the full endomorphism ring, consider its subrings of finitary
linear maps. However, the rings we consider are, in the infinite-dimensional 
case, rather ``small'', significantly ``smaller'' even than the ring of all 
finitary maps, so this can be viewed as a substantial extension of the original
Baer theorem. Moreover, our rings are, generally, not dense, thus all the 
previous methods of proofs of Baer's theorem and its variants, based on 
consideration of idempotents or other structural gadgets from ring theory, or on
Jacobson's density theorem, do not work. Instead, we use an elementary 
linear-algebraic technique of ``decomposing'' conditions imposed on linear maps
on tensor products, and consideration of traces.

\bigskip

All rings in this note are assumed to be associative. For the standard linear 
algebra over a (noncommutative) division ring, we refer to \cite{baer} or 
\cite{jacobson}. Some notation and terminology is also borrowed from \cite{lie}
(where things are treated in the Lie-algebraic context).

Let $D$ be a division ring with unit $1$, and $V$ a right vector space over $D$;
then the dual $V^*$ is a left vector space over $D$. A linear map is called 
\emph{finitary} if its kernel has finite codimension (or, what is equivalent, 
its image has finite dimension). Finitary maps have traces, defined in the usual
manner. The set of all finitary linear maps forms a subring $\FEnd_D(V)$ of the
endomorphism ring $\End_D(V)$, and is linearly spanned by 
\emph{infinitesimal transvections} $t_{v,f}: V \to V$, defined as 
$t_{v,f}(u) = vf(u)$, where $v,u \in V$ and $f \in V^*$ (infinitesimal 
transvections are exactly linear maps whose image is one-dimensional). The 
trace of an infinitesimal transvection is determined by the formula 
$\Tr(t_{v,f}) = f(v)$.

The ring $\FEnd_D(V)$ is isomorphic to the ring $V \otimes_D V^*$, with 
multiplication given by
\begin{equation}\label{eq-m}
(v \otimes f) \cdot (u \otimes g) = vf(u) \otimes g ,
\end{equation}
where $v,u \in V$, $f,g \in V^*$. The isomorphism is given by sending the
infinitesimal transvection $t_{v,f}$ to the decomposable tensor $v \otimes f$.

Now let $\Pi$ be a $D$-subspace of $V^*$. Then $V \otimes_D \Pi$ is still closed
with respect to the multiplication (\ref{eq-m}), and hence forms a subring of 
$V \otimes_D V^*$. Going back to $\End_D(V)$ and its subrings, the ring 
$V \otimes_D \Pi$ is isomorphic to the subring $\FEnd_D(V,\Pi)$ of $\FEnd_D(V)$
generated by all infinitesimal transvections $t_{v,f}$ with $v \in V$ and $f \in \Pi$.

Note that all the just mentioned rings, $\End_D(V)$, $\FEnd_D(V)$, and 
$\FEnd_D(V,\Pi)$, are also right vector spaces over $D$, and hence have a 
structure of a right $D$-algebra.

\begin{theorem}
Let $V, W$ be right vector spaces over a division ring $D$, $\Pi$ a nonzero
finite-dimensional subspace of $V^*$, $\Gamma$ a finite-dimensional subspace of
$W^*$, and $\Phi: \FEnd_D(V,\Pi) \to \FEnd_D(W,\Gamma)$ an isomorphism of 
$D$-algebras. Then there is an isomorphism of $D$-vector spaces 
$\alpha: V \to W$ such that
\begin{equation}\label{eq-conj}
\Phi(f) = \alpha \circ f \circ \alpha^{-1}
\end{equation}
for any $f \in \FEnd_D(V,\Pi)$.
\end{theorem}

\begin{proof}
Write the rings $\FEnd_D(V,\Pi)$ and $\FEnd_D(W,\Gamma)$ in the isomorphic form
as the tensor products $V \otimes_D \Pi$ and $W \otimes_D \Gamma$ as above, and,
by abuse of notation, denote by the same symbol $\Phi$ the isomorphism of 
$D$-algebras $\Phi: V \otimes_D \Pi \to W \otimes_D \Gamma$. Due to 
finite-dimensionality of $\Pi$, we have isomorphism of vector spaces over $D$:
$$
\Hom_D(V \otimes_D \Pi, W \otimes_D \Gamma) \simeq 
\Hom_D(V,W) \otimes_D \Hom_D(\Pi,\Gamma) ,
$$
and hence we can write $\Phi$ in the form 
\begin{equation}\label{eq-d}
\Phi(v \otimes f) = \sum_{i\in I} \alpha_i(v) \otimes \beta_i(f) ,
\end{equation}
where $\alpha_i: V \to W$ and $\beta_i: \Pi \to \Gamma$ are two linearly 
independent families of $D$-linear maps, indexed by a finite set $I$. The 
condition that $\Phi$ is a homomorphism, written for a pair of decomposable 
tensors $v \otimes f$ and $u \otimes g$, where $v,u \in V, f,g \in \Pi$, is 
equivalent to
$$
\sum_{i\in I} \Big(
\alpha_i(v)f(u) -
\sum_{j\in I} \alpha_j(v) \beta_j(f)(\alpha_i(u))\Big) \otimes \beta_i(g)
= 0 .
$$

Since the family $\{\beta_i\}_{i\in I}$ is linearly independent over $D$, each 
first tensor factor in the external sum vanishes, i.e.,
$$
\alpha_i(v)f(u) - \sum_{j\in I} \alpha_j(v) \beta_j(f)(\alpha_i(u)) = 0 
$$
for any $i\in I$, $v,u \in V$, and $f\in \Pi$. This can be rewritten as
$$
\alpha_i(v)\Big(f(u) - \beta_i(f)(\alpha_i(u))\Big)
- \sum_{\substack{j\in I \\ j \ne i}} \alpha_j(v) \beta_j(f)(\alpha_i(u)) = 0 .
$$

Since the family $\{\alpha_i\}_{i\in I}$ is linearly independent over $D$, each 
coefficient from $D$ in the last sum vanishes; in particular, 
\begin{equation}\label{eq-fai}
\beta_i(f)(\alpha_i(u)) = f(u)
\end{equation}
for any $u\in V, f \in \Pi$, and $i\in I$.

This implies
\begin{multline*}
\Tr \Big(\Phi(u \otimes f)\Big) = 
\Tr \Big(\sum_{i\in I} \alpha_i(u) \otimes \beta_i(f)\Big) =
\sum_{i\in I} \Tr\Big(\alpha_i(u) \otimes \beta_i(f)\Big) \\ =
\sum_{i\in I} \beta_i(f)(\alpha_i(u)) = \sum_{i\in I} f(u) = |I| f(u) = 
|I| \Tr (u \otimes f) .
\end{multline*}

As the ring $V \otimes_D \Pi$ is linearly spanned by decomposable tensors, we
have
\begin{equation}\label{eq-tr}
\Tr(\Phi(\xi)) = |I|\Tr(\xi)
\end{equation}
for any $\xi \in V \otimes_D \Pi$. 

Now consider the inverse isomorphism 
$\Phi^{-1}: W \otimes_D \Gamma \to V \otimes_D \Pi$, with decomposition similar
to (\ref{eq-d}) with the index set $J$. By the same reasoning as in the case of
$\Phi$, we have
\begin{equation}\label{eq-tr1}
\Tr(\Phi^{-1}(\eta)) = |J|\Tr(\eta)
\end{equation}
for any $\eta \in W \otimes_D \Gamma$. Combining (\ref{eq-tr}) and 
(\ref{eq-tr1}), we have
\begin{equation}\label{eq-ij}
\Tr(\xi) = |I||J| \Tr(\xi)
\end{equation}
for any $\xi \in V \otimes_D \Pi$. 

Since for any nonzero $f \in V^*$ there is $v \in V$ such that $f(v) \ne 0$ 
(actually, we can choose $f(v)$ to be equal to $1$), $\FEnd_D(V,\Pi)$ always 
contains elements of nonzero trace\footnote{
This is a trivial, albeit a crucial point in our reasoning. Compare with the
condition of so-called totality in 
\cite[Chapter IX, \S 11, Theorem 7]{jacobson}, which, in our notation, amounts 
to saying, in a sense, a dual thing: for any nonzero $v \in V$ there is 
$f \in \Pi$ such that $f(v) \ne 0$. The latter condition is equivalent to the 
density of rings under consideration, and allows to use Jacobson's density 
theorem.
}
. Picking such an element $\xi$ in 
(\ref{eq-ij}), we have $|I| = |J| = 1$, i.e., $\Phi$ preserves traces and can be
represented as a decomposable linear map: $\Phi = \alpha \otimes \beta$ for some
$\alpha: V \to W$ and $\beta: \Pi \to \Gamma$. The system of equalities 
(\ref{eq-fai}) reduces to the single equality
\begin{equation}\label{eq-fa}
\beta(f) \circ \alpha = f
\end{equation}
for any $f \in \Pi$. As $\Phi$ is invertible, $\alpha$ and $\beta$ are 
invertible with $\Phi^{-1} = \alpha^{-1} \otimes \beta^{-1}$, so (\ref{eq-fa}) 
can be rewritten as $\beta(f) = f \circ \alpha^{-1}$, and hence 
$\Phi(v \otimes f) = \alpha(v) \otimes (f \circ \alpha^{-1})$. Rewriting the
last equality in terms of $\FEnd_D(V,\Pi)$ for an infinitesimal transvection
$t_{v,f}$, and expanding by linearity, we get (\ref{eq-conj}). 
\end{proof}

\bigskip

A couple of final remarks concerning possible extensions of the theorem:

\begin{enumerate}[\upshape(i)]
\item
We require the base ring $D$ to be a division ring in order to ensure that all 
$D$-modules are free. We can merely require that all the modules appearing in 
the formulation of the theorem, i.e., $V$, $W$, $\Pi$, $\Gamma$, as well as all 
the modules appearing in the course of the proof, are free, without imposing any
conditions on $D$, but this will just lead to a cumbersome formulation without 
changing the essence of the things.

\item
The theorem can be easily extended to the graded case (thus providing a 
``finitary'' analog of results from \cite{racine} and \cite{balaba}), with 
essentially the same proof which will keep track of the maps on each graded 
component separately. This is left as an exercise to the reader.
\end{enumerate}


\bigskip

Thanks are due to the anonymous referee for indicating an erroneous remark in 
the previous version of the manuscript.

\end{document}